\documentclass[a4paper,11pt,reqno]{amsart}
\usepackage{latexsym}
\usepackage{amssymb,amsfonts,amsmath,mathrsfs}
\newtheorem{theorem}{Theorem}[section]

\newtheorem{lemma}{Lemma}[section]
\newtheorem{remark}{Remark}[section]

\author{H. M. Bui}
\address{Institut f\"ur Mathematik, Universit\"at Z\"urich, Z\"urich CH-8057, Switzerland}
\email{hung.bui@math.uzh.ch}

\title{Large gaps between consecutive zeros of the Riemann zeta-function. II}
\begin{document}
\begin{abstract}
Assuming the Riemann Hypothesis we show that there exist infinitely many consecutive zeros of the Riemann zeta-function whose gaps are greater than $2.9$ times the average spacing.
\end{abstract}


\thanks{The author is supported by the University of Zurich Forschungskredit.}
\maketitle

\section{Introduction}

Subject to the truth of the Riemann Hypothesis (RH), the nontrivial zeros of the Riemann zeta-function can be written as $\rho=\tfrac{1}{2}+i\gamma$, where $\gamma\in\mathbb{R}$. Denote consecutive ordinates of zeros by $0<\gamma\leq\gamma'$, we define the normalized gap
\begin{equation*}
\delta(\gamma):=(\gamma'-\gamma)\frac{\log\gamma}{2\pi}.
\end{equation*}
It is well-known that
\begin{displaymath}
N(T):=\sum_{0<\gamma\leq T}1=\frac{T}{2\pi}\log\frac{T}{2\pi}-\frac{T}{2\pi}+O(\log T)
\end{displaymath}
for $T\geq 10$. Hence $\delta(\gamma)$ is $1$ on average. It is expected that there are arbitrarily large and arbitrarily small (normalized) gaps between consecutive zeros of the Riemann zeta-function on the critical line, i.e.
\begin{equation*}
\lambda:=\limsup_{\gamma}\delta(\gamma)=\infty\quad\textrm{and}\ \ \mu:=\liminf_{\gamma}\delta(\gamma)=0.
\end{equation*}
In this article, we focus only on the large gaps, and prove the following theorem.

\begin{theorem}
Assuming RH. Then we have $\lambda>2.9$.
\end{theorem}

Very little is known about $\lambda$ unconditionally. Selberg [\textbf{\ref{S}}] remarked that he could prove $\lambda>1$. Conditionally, Bredberg [\textbf{\ref{B1}}] showed that $\lambda>2.766$ under the assumption of RH (see also [\textbf{\ref{M}},\textbf{\ref{MO}},\textbf{\ref{CGG3}},\textbf{\ref{H}},\textbf{\ref{BMN}},\textbf{\ref{FW1}}] for work in this direction), and on the Generalized Riemann Hypothesis (GRH) it is known that $\lambda>3.072$ [\textbf{\ref{FW}}] (see also [\textbf{\ref{CGG2}},\textbf{\ref{Ng}},\textbf{\ref{B}}]). These results either use Hall's approach using Wirtinger's inequality, or exploit the following idea of Mueller [\textbf{\ref{M}}].

Let $H:\mathbb{C}\rightarrow\mathbb{C}$ and consider the following functions
\begin{equation*}
\mathcal{M}_1(H,T)=\int_{0}^{T}\big|H(\tfrac{1}{2}+it)\big|^2dt
\end{equation*}
and
\begin{equation*}
\mathcal{M}_2(H,T;c)=\int_{-c/L}^{c/L}\sum_{0<\gamma\leq T}\big|H(\tfrac{1}{2}+i(\gamma+\alpha))\big|^2d\alpha,
\end{equation*}
where $L=\log\frac{T}{2\pi}$. We note that if
\begin{equation*}
h(c):=\frac{\mathcal{M}_2(H,T;c)}{\mathcal{M}_1(H,T)}<1
\end{equation*}
as $T\rightarrow\infty$, then $\lambda>c/\pi$, and if $h(c)>1$ as $T\rightarrow\infty$, then $\mu<c/\pi$.

Mueller [\textbf{\ref{M}}] applied this idea to $H(s)=\zeta(s)$. Using $H(s)=\sum_{n\leq T^{1-\varepsilon}}d_{2.2}(n)n^{-s}$, where
the arithmetic function $d_k(n)$ is defined in terms of the Dirichlet series
\begin{equation*}
\zeta(s)^k=\sum_{n=1}^{\infty}\frac{d_k(n)}{n^s}\qquad(\sigma>1)
\end{equation*}
for any real number $k$, Conrey, Ghosh and Gonek [\textbf{\ref{CGG1}}] showed that $\lambda>2.337$. Later, assuming GRH, they applied to $H(s)=\zeta(s)\sum_{n\leq T^{1/2-\varepsilon}}n^{-s}$ and obtained $\lambda>2.68$ [\textbf{\ref{CGG2}}]. By considering a more general choice
\begin{equation*}
H(s)=\zeta(s)\sum_{n\leq T^{1/2-\varepsilon}}\frac{d_r(n)P(\frac{\log y/n}{\log y})}{n^s},
\end{equation*}
where $P(x)$ is a polynomial, Ng [\textbf{\ref{Ng}}] improved that result to $\lambda>3$ (using $r=2$ and $P(x)=(1-x)^{30}$). In the last two papers, GRH is needed to estimate some certain exponential sums resulting from the evaluation of the discrete mean value over the zeros in $\mathcal{M}_2(H,T;c)$. Recently, Bui and Heath-Brown [\textbf{\ref{BH-B}}] showed how one can use a generalization of the Vaughan identity and the hybrid large sieve inequality to circumvent the assumption of GRH for such exponential sums. Here we use that idea to obtain a weaker version of Ng's result without provoking GRH. It is possible that Feng and Wu's result $\lambda>3.072$ can also be obtained just assuming RH by this method. However, we opt to work on Ng's result for simplicity.

Instead of using the divisor function $d(n)=d_2(n)$, we choose
\begin{equation*}
H(s)=\zeta(s)\sum_{n\leq y}\frac{h(n)P(\frac{\log y/n}{\log y})}{n^s},
\end{equation*}
where $y=T^{\vartheta}$, $P(x)$ is a polynomial and $h(n)$ is a multiplicative function satisfying
\begin{equation}\label{500}
h(n)=\left\{ \begin{array}{ll}
d(n) &\qquad \textrm{if $n$ is square-free,}\\
0 & \qquad\textrm{otherwise.} 
\end{array} \right.
\end{equation}
In Section 3 and Section 4 we shall prove the following two key lemmas.

\begin{lemma}
Suppose $0<\vartheta<\tfrac{1}{2}$. We have
\begin{eqnarray*}
\mathcal{M}_1(H,T)=\frac{AT(\log y)^{9}}{6}\int_{0}^{1}(1-x)^{3}\bigg(\vartheta^{-1}P_{1}(x)^2-2P_{1}(x)P_{2}(x)\bigg)dx+O(TL^8),
\end{eqnarray*}
where
\begin{equation*}
A=\prod_{p}\bigg(1+\frac{8}{p}\bigg)\bigg(1-\frac{1}{p}\bigg)^{8}
\end{equation*}
and
\begin{equation*}
P_{r}(x)=\int_{0}^{x}t^{r}P(x-t)dt.
\end{equation*}
\end{lemma}

\begin{lemma}
Suppose $0<\vartheta<\tfrac{1}{2}$ and $P(0)=P'(0)=0$. We have
\begin{eqnarray*}
\sum_{0<\gamma\leq T}H(\rho+i\alpha)H(1-\rho-i\alpha)&=&\frac{ATL(\log y)^{9}}{6\pi}\int_{0}^{1}(1-x)^{3}\emph{Re}\bigg\{{\sum_{j=1}^{\infty}(i\alpha\log y)^jB(j;x)}\bigg\}dx\\
&&\qquad+O_\varepsilon(TL^{9+\varepsilon})
\end{eqnarray*}
uniformly for $\alpha\ll L^{-1}$, where
\begin{eqnarray*}
&&B(j;u)=-\frac{2P_1(u)P_{j+2}(u)}{(j+2)!}+\frac{2\vartheta P_2(u)P_{j+2}(u)}{(j+2)!}+\frac{4\vartheta P_1(u)P_{j+3}(u)}{(j+3)!}\nonumber\\
&&\qquad-\frac{\vartheta}{(j+2)!}\int_{0}^{u}t(\vartheta^{-1}-t)^{j+2}P_{1}(u)P(u-t)dt\\
&&\qquad+\frac{\vartheta}{(j+1)!}\int_{0}^{u}t(\vartheta^{-1}-t)^{j+1}P_{2}(u)P(u-t)dt-\frac{\vartheta}{6j!}\int_{0}^{u}t(\vartheta^{-1}-t)^{j}P_{3}(u)P(u-t)dt.
\end{eqnarray*}
\end{lemma}

\noindent\textit{Proof of} Theorem 1.1. We take $\vartheta=\tfrac{1}{2}^{-}$. On RH we have
\begin{equation*}
\sum_{0<\gamma\leq T}\big|H(\tfrac{1}{2}+i(\gamma+\alpha))\big|^2=\sum_{0<\gamma\leq T}H(\rho+i\alpha)H(1-\rho-i\alpha).
\end{equation*}
Note that this is the only place we need to assume RH. Lemma 1.2 then implies that
\begin{eqnarray*}
\int_{-c/L}^{c/L}\sum_{0<\gamma\leq T}\big|H(\tfrac{1}{2}+i(\gamma+\alpha))\big|^2d\alpha\sim\frac{AT(\log y)^{9}}{6\pi}\sum_{j=1}^{\infty}\frac{(-1)^jc^{2j+1}}{2^{2j-1}(2j+1)}\int_{0}^{1}(1-x)^{3}B(2j;x)dx.
\end{eqnarray*}
Hence
\begin{eqnarray*}
h(c)=\frac{1}{2\pi}\frac{\sum_{j=1}^{\infty}\frac{(-1)^jc^{2j+1}}{2^{2j-1}(2j+1)}\int_{0}^{1}(1-x)^{3}B(2j;x)dx}{\int_{0}^{1}(1-x)^{3}(P_1(x)^2-P_1(x)P_2(x))dx}+o(1),
\end{eqnarray*}
as $T\rightarrow\infty$. Consider the polynomial $P(x)=\sum_{j=2}^{M}c_jx^j$. Choosing $M=6$ and running Mathematica's Minimize command, we obtain $\lambda>2.9$. Precisely, with 
\begin{eqnarray*}
P(x)=1000x^2 - 9332x^3 + 30134x^4 - 40475x^5 + 19292x^6,
\end{eqnarray*}
we have
\begin{eqnarray*}
h(2.9\pi)=0.99725\ldots<1,
\end{eqnarray*}
and this proves the theorem.

\begin{remark} 
\emph{The above lemmas are unconditional. We note that in the case $r=2$ apart from the arithmetical factor $a_3$ being replaced by $A$, Lemma 1.1 is the same as what stated in [\textbf{\ref{Ng}}; Lemma 2.1] (see also [\textbf{\ref{B}}; Lemma 2.3]), while Lemma 1.2, under the additional condition $P(0)=P'(0)=0$, recovers Theorem 2 of Ng [\textbf{\ref{Ng}}] (and also Lemma 2.6 of Bui [\textbf{\ref{B}}]) without assuming GRH, though the latters are written in a slightly different and more complicated form. This is as expected because replacing the divisor function $d(n)$ by the arithmetic function $h(n)$ (as defined in \eqref{500}) in the definition of $H(s)$ only changes the arithmetical factor in the resulting mean value estimates. This substitution, however, makes our subsequent calculations much easier. Our arguments also work if we set $h(n)=d_r(n)$ when $n$ is square-free for some $r\in\mathbb{N}$ without much changes, but we choose $r=2$ to simplify various statements and expressions in the paper.}
\end{remark}

\begin{remark}
\emph{In the course of evaluating $\mathcal{M}_2(H,T;c)$, we encounter an exponential sum of type (see Section 4.2)
\begin{equation*}
\sum_{n\leq y}\frac{h(n)P(\frac{\log y/n}{\log y})}{n}\sum_{m\leq nT/2\pi}a(m)e\bigg(-\frac{m}{n}\bigg)
\end{equation*}
for some arithmetic function $a(m)$. At this point, assuming GRH, Ng [\textbf{\ref{Ng}}] applied Perron's formula to the sum over $m$, and then moved the line of integration to $\textrm{Re}(s)=1/2+\varepsilon$. The main term arises from the residue at $s=1$ and the error terms in this case are easy to handle. To avoid being subject to GRH, we instead use the ideas in [\textbf{\ref{CGG1}}] and [\textbf{\ref{BH-B}}]. That leads to a sum of type
\begin{equation*}
\sum_{n\leq y}\frac{\mu(n)h(n)P(\frac{\log y/n}{\log y})}{n}.
\end{equation*}
This is essentially a variation of the prime number theorem, and here the polynomial $P(x)$ is required to vanish with order at least $2$ at $x=0$ (see Lemma 2.6). As a result, we cannot take the choice $P(x)=(1-x)^{30}$ as in [\textbf{\ref{Ng}}]. Here it is not clear how to choose a ``good" polynomial $P(x)$. Our theorem is obtained by numerically optimizing over polynomials $P(x)$ with degree less than $7$. It is probable that by considering higher degree polynomials, we can establish Ng's result $\lambda>3$ under only RH.}
\end{remark}

\noindent\textbf{Notation.} Throughout the paper, we denote
\begin{equation*}
[n]_y:=\frac{\log y/n}{\log y}.
\end{equation*}
For $Q,R\in C^\infty[(0,1)]$ we define
\begin{equation*}
Q_{r}(x)=\int_{0}^{x}t^{r}Q(x-t)dt\qquad\textrm{and}\qquad R_{r}(x)=\int_{0}^{x}t^{r}R(x-t)dt.
\end{equation*}
We let $\varepsilon>0$ be an arbitrarily small positive number, and can change from time to time.

\section{Various lemmas}

The following two lemmas are in [\textbf{\ref{CGG1}}; Lemma 2 and Lemma 3].

\begin{lemma}\label{501}
Suppose that $A(s)=\sum_{m=1}^{\infty}a(m)m^{-s}$, where $a(m)\ll_\varepsilon m^\varepsilon$, and $B(s)=\sum_{n\leq y}b(n)n^{-s}$, where $b(n)\ll_\varepsilon n^\varepsilon$. Then we have
\begin{eqnarray*}
&&\frac{1}{2\pi i}\int_{a+i}^{a+iT}\chi(1-s)A(s)B(1-s)ds=\sum_{n\leq y}\frac{b(n)}{n}\sum_{m\leq nT/2\pi}a(m)e\bigg(-\frac{m}{n}\bigg)+O_\varepsilon(yT^{1/2+\varepsilon}),
\end{eqnarray*} 
where $a=1+L^{-1}$.
\end{lemma}

\begin{lemma}\label{300}
Suppose that $A_j(s)=\sum_{n=1}^{\infty}a_{j}(n)n^{-s}$ is absolutely convergent for $\sigma>1$, $1\leq j\leq k$, and that
\begin{equation*}
A(s)=\sum_{n=1}^{\infty}\frac{a(n)}{n^s}=\prod_{j=1}^{k}A_j(s).
\end{equation*}
Then for any $l\in\mathbb{N}$, we have
\begin{equation*}
\sum_{n=1}^{\infty}\frac{a(ln)}{n^s}=\sum_{l=l_1\ldots l_k}\prod_{j=1}^{k}\bigg(\sum_{\substack{n\geq1\\(n,\prod_{i<j}l_i)=1}}\frac{a_{j}(l_jn)}{n^s}\bigg).
\end{equation*}
\end{lemma}

We shall need estimates for various divisor-like sums. Throughout the paper, we let 
\begin{equation*}
F_\tau(n)=\prod_{p|n}\big(1+O(p^{-\tau})\big),
\end{equation*}
for $\tau>0$ and the constant in the $O$-term is implicit and independent of $\tau$.

\begin{lemma}\label{504}
For any $Q\in C^\infty([0,1])$, there exists an absolute constant $\tau_0>0$ such that
\begin{eqnarray*}
&&\emph{(i)}\ \sum_{an\leq y}\frac{h(an)Q([an]_y)}{n}=C(\log y)^2h(a)\prod_{p|a}\bigg(1+\frac{2}{p}\bigg)^{-1}Q_1([a]_y)+O(d(a)F_{\tau_0}(a)L),\\
&&\emph{(ii)}\ \sum_{an\leq y}\frac{h(an)Q([an]_y)\log n}{n}=C(\log y)^3h(a)\prod_{p|a}\bigg(1+\frac{2}{p}\bigg)^{-1}Q_2([a]_y)+O(d(a)F_{\tau_0}(a)L^2),
\end{eqnarray*}
where
\begin{equation*}
C=\prod_{p}\bigg(1+\frac{2}{p}\bigg)\bigg(1-\frac{1}{p}\bigg)^2.
\end{equation*}
\end{lemma}
\begin{proof}
By a method of Selberg [\textbf{\ref{S}}] we have
\begin{equation*}
\sum_{n\leq t}\frac{h(an)}{n}=\frac{C(\log t)^2}{2}h(a)\prod_{p|a}\bigg(1+\frac{2}{p}\bigg)^{-1}+O(d(a)F_{\tau_0}(a)L)
\end{equation*}
for any $t\leq T$. The first statement then follows from partial summation. 

The second statement is an easy consequence of the first one.
\end{proof}

\begin{lemma}\label{507}
For any $Q\in C^\infty([0,1])$, we have
\begin{equation*}
\sum_{n\leq y}\frac{h(n)^2\varphi(n)Q([n]_y)}{n^2}\prod_{p|n}\bigg(1+\frac{2}{p}\bigg)^{-2}=\frac{D(\log y)^4}{6}\int_{0}^{1}(1-x)^3Q(x)dx+O(L^3),
\end{equation*}
where
\begin{equation*}
D=\prod_p\bigg[1+\frac{4(p-1)}{p^2}\bigg(1+\frac{2}{p}\bigg)^{-2}\bigg]\bigg(1-\frac{1}{p}\bigg)^4.
\end{equation*}
\end{lemma}
\begin{proof}
The proof is similar to the above lemma. 
\end{proof}

We need a lemma concerning the size of the function $F_{\tau_0}(n)$ on average.

\begin{lemma}\label{505}
Suppose $-1\leq\sigma\leq 0$. We have
\begin{equation*}
\sum_{n\leq y} \frac{d_k(n)F_{\tau_0}(n)}{n}\bigg(\frac{y}{n}\bigg)^{\sigma}\ll_k L^{k-1} \min\big\{|\sigma|^{-1},L\big\}.
\end{equation*}
\end{lemma}
\begin{proof}
We use Lemma 4.6 in [\textbf{\ref{BCY}}] that
\begin{equation*}
\sum_{n\leq y} \frac{d_k(n)}{n}\bigg(\frac{y}{n}\bigg)^{\sigma}\ll_k L^{k-1} \min\big\{|\sigma|^{-1},L\big\}.
\end{equation*}
We have
\begin{equation*}
F_{\tau_0}(n)\leq\prod_{p|n}\big(1+Ap^{-\tau_0}\big)=\sum_{l|n}l^{-\tau_0}A^{w(l)}
\end{equation*}
for some $A>0$, where $w(n)$ is the number of prime factors of $n$. Hence
\begin{eqnarray*}
\sum_{n\leq y} \frac{d_k(n)F_{\tau_0}(n)}{n}\bigg(\frac{y}{n}\bigg)^{\sigma}\ll\sum_{l\leq y}\frac{d_k(l)A^{w(l)}}{l^{1+\tau_0}}\sum_{n\leq y/l}\frac{d_k(n)}{n}\bigg(\frac{y/l}{n}\bigg)^{\sigma}\ll_k L^{k-1} \min\big\{|\sigma|^{-1},L\big\},
\end{eqnarray*}
since $d_k(l)A^{w(l)}\ll l^{\tau_0/2}$ for sufficiently large $l$.
\end{proof}

\begin{lemma}\label{600}
Let $F(n)=F(n,0)$, where
\begin{equation*}
F(n,\alpha)=\prod_{p|n}\bigg(1-\frac{1}{p^{1+\alpha}}\bigg).
\end{equation*}
For any $Q\in C^\infty([0,1])$ satisfying $Q(0)=Q'(0)=0$, there exist an absolute constant $\tau_0>0$ and some $\nu\asymp (\log\log y)^{-1}$ such that
\begin{eqnarray*}
\mathcal{A}_1(y,Q;a,b,\underline{\alpha})&=&\sum_{\substack{an\leq y\\(n,b)=1}}\frac{\mu(n)h(n)Q([an]_y)}{\varphi(n)n^{\alpha_1}}F(n,\alpha_2)F(n,\alpha_3)\\
&=&U_1V_1(b)\bigg(\frac{Q''([a]_y)}{(\log y)^2}+\frac{2\alpha_1 Q'([a]_y)}{\log y}+\alpha_{1}^{2}Q([a]_y)\bigg)\\
&&\qquad+O(F_{\tau_0}(b)L^{-3})+O_\varepsilon\bigg(F_{\tau_0}(b)\bigg(\frac{y}{a}\bigg)^{-\nu} L^{-2+\varepsilon}\bigg)
\end{eqnarray*}
uniformly for $\alpha_j\ll L^{-1}$, $1\leq j\leq 3$, where $U_1=U_1(0,\underline{0})$ and $V_1(n)=V_1(0,n,\underline{0})$, with
\begin{equation*}
U_1(s,\underline{\alpha})=\prod_{p}\bigg[1-\frac{2F(p,\alpha_2)F(p,\alpha_3)}{\varphi(p)p^{s+\alpha_1}}\bigg]\bigg(1-\frac{1}{p^{1+s+\alpha_1}}\bigg)^{-2}
\end{equation*}
and
\begin{equation*}
V_1(s,n,\underline{\alpha})=\prod_{p|n}\bigg[1-\frac{2F(p,\alpha_2)F(p,\alpha_3)}{\varphi(p)p^{s+\alpha_1}}\bigg]^{-1}.
\end{equation*}
\end{lemma}
\begin{proof}
This is essentially a variation of the prime number theorem. 

It suffices to consider $Q(x)=\sum_{j\geq 2}a_jx^j$. We have
\begin{equation*}
\mathcal{A}_1(y,Q;a,b,\underline{\alpha})=\sum_{j\geq 2}\frac{a_jj!}{(\log y)^j}\sum_{(n,b)=1}\frac{1}{2\pi i}\int_{(2)}\bigg(\frac{y}{a}\bigg)^s\frac{\mu(n)h(n)}{\varphi(n)n^{s+\alpha_1}}F(n,\alpha_2)F(n,\alpha_3)\frac{ds}{s^{j+1}}.
\end{equation*}
The sum over $n$ converges absolutely. Hence
\begin{equation*}
\mathcal{A}_1(y,Q;a,b,\underline{\alpha})=\sum_{j\geq 2}\frac{a_jj!}{(\log y)^j}\frac{1}{2\pi i}\int_{(2)}\bigg(\frac{y}{a}\bigg)^s\sum_{(n,b)=1}\frac{\mu(n)h(n)}{\varphi(n)n^{s+\alpha_1}}F(n,\alpha_2)F(n,\alpha_3)\frac{ds}{s^{j+1}}.
\end{equation*}
The sum in the integrand equals
\begin{equation*}
\prod_{p\nmid b}\bigg(1-\frac{2F(p,\alpha_2)F(p,\alpha_3)}{\varphi(p)p^{s+\alpha_1}}\bigg)=\frac{U_1(s,\underline{\alpha})V_1(s,b,\underline{\alpha})}{\zeta(1+s+\alpha_1)^2}.
\end{equation*}

Let $Y = o(T)$ be a large parameter to be chosen later.  By Cauchy's theorem, $\mathcal{A}_1(y,Q;a,b,\underline{\alpha})$ is equal to the residue at $s=0$ plus integrals over the line segments $\mathcal{C}_1=\{s=it,t\in\mathbb{R},|t|\geq Y\}$, $\mathcal{C}_{2}=\{s=\sigma\pm iY,-\frac{c}{\log{Y}}\leq\sigma\leq 0\}$, and $\mathcal{C}_3=\{s=-\frac{c}{\log{Y}}+it,|t|\leq Y\}$, where $c$ is some fixed positive constant such that $\zeta(1+ s+\alpha_1)$ has no zeros in the region on the right hand side of the contour determined by the $\mathcal{C}_j$'s. Furthermore, we require that for such $c$ we have $1/\zeta(\sigma + it) \ll \log(2 + |t|)$ in this region [see \textbf{\ref{T}}; Theorem 3.11].  Then the integral over $\mathcal{C}_1$ is 
\begin{equation*}
\ll F_{\tau_0}(b)L^{-j}(\log{Y})^2/Y^{j} \ll_\varepsilon F_{\tau_0}(b)L^{-2}Y^{-2+\varepsilon},
\end{equation*} 
since $j\geq2$. The integral over $\mathcal{C}_2$ is 
\begin{equation*}
\ll F_{\tau_0}(b)L^{-j}(\log{Y})/Y^{j+1} \ll_\varepsilon F_{\tau_0}(b)L^{-2}Y^{-3+\varepsilon}.
\end{equation*}  
Finally, the contribution from $\mathcal{C}_3$ is 
\begin{equation*}
\ll F_{\tau_0}(b)L^{-j}(\log Y)^j \bigg(\frac{y}{a}\bigg)^{-c/\log Y}\ll_\varepsilon F_{\tau_0}(b)\bigg(\frac{y}{a}\bigg)^{-c/\log Y}L^{-2+\varepsilon}.
\end{equation*}  
Choosing $Y \asymp L$ gives an error so far of size $O_\varepsilon\big(F_{\tau_0}(b)(y/a)^{-\nu} L^{-2+\varepsilon}\big) + O_\varepsilon(F_{\tau_0}(b)L^{-4+\varepsilon})$.

For the residue at $s=0$, we write this as
\begin{equation*}
\sum_{j\geq 2}\frac{a_jj!}{(\log y)^j}\frac{1}{2 \pi i} \oint \bigg(\frac{y}{a}\bigg)^s \frac{U_1(s,\underline{\alpha})V_1(s,b,\underline{\alpha})}{\zeta(1+s+\alpha_1)^2}  \frac{ds}{s^{j+1}},
\end{equation*}
where the contour is a circle of radius $\asymp L^{-1}$ around the origin. This integral is trivially bounded by $O(L^{-2})$ so that taking the first term in the Taylor series of $\zeta(1+s+\alpha_1)$ finishes the proof.
\end{proof}

\begin{lemma}\label{601}
For any $Q,R\in C^\infty([0,1])$, there exists an absolute constant $\tau_0>0$ such that
\begin{eqnarray*}
&&\mathcal{A}_2(y,Q,R;a_1,a_2,\alpha_1)=\sum_{\substack{a_1a_2l\leq y\\a_1m\leq y}}\frac{h(a_1a_2l)h(a_1m)Q([a_1m]_y)R([a_1a_2l]_y)V_1(a_1a_2lm)}{lm^{1+\alpha_1}}\\
&&\qquad\qquad=U_2(\log y)^4h(a_1a_2)h(a_1)V_1(a_1a_2)V_2(a_1)V_3(a_2)V_4(a_1a_2)\\
&&\qquad\qquad\qquad\qquad \int_{0}^{[a_1]_y}y^{-\alpha_1t}tQ([a_1]_y-t)R_1([a_1a_2]_y)dt+O(d_4(a_1)d(a_2)F_{\tau_0}(a_1a_2)L^3)
\end{eqnarray*}
uniformly for $\alpha_1\ll L^{-1}$, where 
\begin{eqnarray*}
U_2=\prod_{p}\bigg(1+\frac{2V_1(p)}{p}\bigg)\bigg[1+\frac{2V_1(p)}{p}\bigg(1+\frac{2}{p}\bigg)\bigg(1+\frac{2V_1(p)}{p}\bigg)^{-1}\bigg]\bigg(1-\frac{1}{p}\bigg)^4,
\end{eqnarray*}
\begin{equation*}
V_2(n)=\prod_{p|n}\bigg(1+\frac{2V_1(p)}{p}\bigg)^{-1},\qquad
V_3(n)=\prod_{p|n}\bigg(1+\frac{2}{p}\bigg)\bigg(1+\frac{2V_1(p)}{p}\bigg)^{-1}
\end{equation*}
and
\begin{eqnarray*}
V_4(n)=\prod_{p|n}\bigg[1+\frac{2V_1(p)}{p}\bigg(1+\frac{2}{p}\bigg)\bigg(1+\frac{2V_1(p)}{p}\bigg)^{-1}\bigg]^{-1}.
\end{eqnarray*}
\end{lemma}
\begin{proof}
The proof uses Selberg's method [\textbf{\ref{S}}] similarly to Lemma \ref{504}. One first executes the sum over $m$, and then the sum over $l$.
\end{proof}

\begin{lemma}\label{602}
For any $Q,R\in C^\infty([0,1])$, we have
\begin{eqnarray*}
&&\emph{(i)}\quad\sum_{l_1l_2\leq y}\frac{h(l_1l_2)h(l_1)Q([l_1]_y)R([l_1l_2]_y)}{l_{1}l_{2}^{1+\alpha_1}} F(l_1,\alpha_2)F(l_1l_2,\alpha_3)V_1(l_1l_2)V_2(l_1)V_3(l_2)V_4(l_1l_2)\\
&&\qquad\qquad=\frac{W(\log y)^6}{6}\int_{0}^{1}\int_{0}^{x}(1-x)^3y^{-\alpha_1t_1}t_1Q(x)R(x-t_1)dt_1dx+O(L^5),\\
&&\emph{(ii)}\quad\sum_{pl_1l_2\leq y}\frac{\log p}{(p^{1+\alpha_4}-1)p^{\alpha_5}}\frac{h(pl_1l_2)h(l_1)Q([l_1]_y)R([pl_1l_2]_y)}{l_{1}l_{2}^{1+\alpha_1}}\\
&&\qquad\qquad\qquad\qquad F(pl_1,\alpha_2)F(pl_1l_2,\alpha_3)V_1(pl_1l_2)V_2(l_1)V_3(pl_2)V_4(pl_1l_2)\\
&&\qquad\qquad=\frac{W(\log y)^7}{3}\int_{0}^{1}\int_{\substack{t_j\geq0\\t_1+t_2\leq x}}(1-x)^3y^{-\alpha_1t_1-(\alpha_4+\alpha_5)t_2}t_1Q(x)R(x-t_1-t_2)dt_1dt_2dx\\
&&\qquad\qquad\qquad\qquad+O(L^6)
\end{eqnarray*}
uniformly for $\alpha_j\ll L^{-1}$, $1\leq j\leq5$, where 
\begin{eqnarray*}
W=\prod_{p}\bigg(1+\frac{2F(p)V_1(p)V_3(p)V_4(p)}{p}+\frac{4F(p)^2V_1(p)V_2(p)V_4(p)}{p}\bigg)\bigg(1-\frac{1}{p}\bigg)^6.
\end{eqnarray*}
\end{lemma}
\begin{proof}
We consider the first statement. We start with the sum over $l_2$ on the left hand side of (i), which is
\begin{equation*}
\sum_{\substack{l_2\leq y/l_1\\(l_2,l_1)=1}}\frac{h(l_2)R([l_1l_2]_y)}{l_{2}^{1+\alpha_1}}F(l_2,\alpha_3)V_1(l_2)V_3(l_2)V_4(l_2).
\end{equation*}
As in how we prove Lemma \ref{504}, this equals
\begin{eqnarray}\label{701}
\prod_{p}\bigg\{W_1(p)^{-1}\bigg(1-\frac{1}{p}\bigg)^2\bigg\}(\log y)^2W_1(l_1)\int_{0}^{[l_1]_y}y^{-\alpha_1t_1}t_1R([l_1]_y-t_1)dt_1+O(L),
\end{eqnarray}
where
\begin{equation*}
W_1(n)=\prod_{p|n}\bigg(1+\frac{2F(p)V_1(p)V_3(p)V_4(p)}{p}\bigg)^{-1}.
\end{equation*}
Hence the required expression is
\begin{eqnarray}\label{700}
&&\prod_{p}\bigg\{W_1(p)^{-1}\bigg(1-\frac{1}{p}\bigg)^2\bigg\}(\log y)^2\sum_{l_1\leq y}\frac{h(l_1)^2Q([l_1]_y)}{l_{1}}\\
&&\qquad F(l_1,\alpha_2)F(l_1,\alpha_3)V_1(l_1)V_2(l_1)V_4(l_1)W_1(l_1)\int_{0}^{[l_1]_y}y^{-\alpha_1t_1}t_1R([l_1]_y-t_1)dt_1+O(L^5).\nonumber
\end{eqnarray}
Using Selberg's method [\textbf{\ref{S}}] again we have
\begin{eqnarray*}
&&\sum_{l_1\leq t}\frac{h(l_1)^2}{l_{1}}F(l_1,\alpha_2)F(l_1,\alpha_3)V_1(l_1)V_2(l_1)V_4(l_1)W_1(l_1)\\
&&\qquad\qquad=\prod_p\bigg\{W_2(p)^{-1}\bigg(1-\frac{1}{p}\bigg)^4\bigg\}\frac{(\log t)^4}{24}+O(L^3)
\end{eqnarray*}
for any $t\leq T$, where
\begin{equation*}
W_2(n)=\prod_{p|n}\bigg\{1+\frac{4F(p)^2V_1(p)V_2(p)V_4(p)W_1(p)}{p}\bigg\}^{-1}.
\end{equation*} 
Partial summation then implies that \eqref{700} is equal to
\begin{eqnarray*}
\prod_p\bigg\{W_1(p)^{-1}W_2(p)^{-1}\bigg(1-\frac{1}{p}\bigg)^6\bigg\}\frac{(\log y)^4}{6}\int_{0}^{1}\int_{0}^{x}(1-x)^3y^{-\alpha_1t_1}t_1Q(x)R(x-t_1)dt_1dx+O(L^5).
\end{eqnarray*}
It is easy to check that the arithmetical factor is $W$, and we obtain the first statement.

For the second statement, we first notice that the contribution of the terms involving $p^{-s}$ with $\textrm{Re}(s)>1$ is $O(L^6)$. Hence the left hand side of (ii) is
\begin{eqnarray*}
&&2\sum_{l_1l_2\leq y}\frac{h(l_1l_2)h(l_1)Q([l_1]_y)}{l_{1}l_{2}^{1+\alpha_1}}F(l_1,\alpha_2)F(l_1l_2,\alpha_3)V_1(l_1l_2)V_2(l_1)V_3(l_2)V_4(l_1l_2)\nonumber\\
&&\qquad\sum_{\substack{p\leq y/l_1l_2\\(p,l_1l_2)=1}}\frac{(\log p)R([pl_1l_2]_y)}{p^{1+\alpha_4+\alpha_5}}+O(L^6).
\end{eqnarray*} 
The same argument shows that we can include the terms $p|l_1l_2$ in the innermost sum with an admissible error $O(L^6)$, so that the above expression is equal to
\begin{eqnarray*}
&&2\sum_{p\leq y}\frac{\log p}{p^{1+\alpha_4+\alpha_5}}\sum_{l_1l_2\leq y/p}\frac{h(l_1l_2)h(l_1)Q([l_1]_y)R([pl_1l_2]_y)}{l_{1}l_{2}^{1+\alpha_1}}\\
&&\qquad\qquad\qquad\qquad\qquad\qquad F(l_1,\alpha_2)F(l_1l_2,\alpha_3)V_1(l_1l_2)V_2(l_1)V_3(l_2)V_4(l_1l_2)+O(L^6).
\end{eqnarray*}
We have
\begin{equation*}
\sum_{p\leq t}\frac{\log p}{p}=\log t+O(1)
\end{equation*}
for any $t\leq T$. The result follows by using Part (i) and partial summation.
\end{proof}

\section{Proof of Lemma 1.1}

To evaluate $\mathcal{M}_1(H,T)$, we first appeal to Theorem 1 of [\textbf{\ref{BCH}}] and obtain
\begin{eqnarray*}
\mathcal{M}_1(H,T)&=&T\sum_{m,n\leq y}\frac{h(m)h(n)P([m]_y)P([n]_y)(m,n)}{mn}\bigg(\log\frac{T(m,n)^2}{2\pi mn}+2\gamma-1\bigg)\\
&&\qquad+O_B(TL^{-B})+O_\varepsilon(y^2T^\varepsilon)
\end{eqnarray*}
for any $B>0$, where $\gamma$ is the Euler constant. Using the M\"obius inversion formula 
\begin{equation*}
f\big((m,n)\big)=\sum_{\substack{l|m\\l|n}}\sum_{d|l}\mu(d)f\bigg(\frac{l}{d}\bigg),
\end{equation*}
we can write the above as
\begin{eqnarray*}
T\sum_{l\leq y}\sum_{d|l}\frac{\mu(d)}{dl}\sum_{m,n\leq y/l}\frac{h(lm)h(ln)P([lm]_y)P([ln]_y)}{mn}\bigg(\log\frac{T}{2\pi d^2mn}+2\gamma-1\bigg)+O_B(TL^{-B}).
\end{eqnarray*}
We next replace the term in the bracket by $\log \frac{T}{2\pi mn}$. This produces an error of size
\begin{eqnarray*}
\ll T\sum_{l\leq y}\frac{d(l)^2}{l}\bigg(\sum_{n\leq y/l}\frac{d(n)}{n}\bigg)^2\sum_{d|l}\frac{\log d}{d}\ll TL^8.
\end{eqnarray*}
Hence
\begin{eqnarray*}
\mathcal{M}_1(H,T)&=&T\sum_{l\leq y}\frac{\varphi(l)}{l^2}\sum_{m,n\leq y/l}\frac{h(lm)h(ln)P([lm]_y)P([ln]_y)}{mn}\big(L-\log m-\log n\big)+O(TL^8)\\
&=&TL\sum_{l\leq y}\frac{\varphi(l)}{l^2}\bigg(\sum_{n\leq y/l}\frac{h(ln)P([ln]_y)}{n}\bigg)^2\\
&&\qquad-2T\sum_{l\leq y}\frac{\varphi(l)}{l^2}\sum_{m,n\leq y/l}\frac{h(lm)h(ln)P([lm]_y)P([ln]_y)\log n}{mn}+O(TL^8).
\end{eqnarray*}
The result follows by using Lemma \ref{504}, Lemma \ref{507} and Lemma \ref{505}. Here we use a fact which is easy to verify that $C^2D=A$.

\section{Proof of Lemma 1.2}

We denote $H(s)=\zeta(s)G(s)$, i.e.
\begin{equation*}
G(s)=\sum_{n\leq y}\frac{h(n)P([n]_y)}{n^s}.
\end{equation*}
By Cauchy's theorem we have
\begin{eqnarray*}
\sum_{0<\gamma\leq T}H(\rho+i\alpha)H(1-\rho-i\alpha)=\frac{1}{2\pi i}\int_{\mathcal{C}}\frac{\zeta'}{\zeta}(s)\zeta(s+i\alpha)\zeta(1-s-i\alpha)G(s+i\alpha)G(1-s-i\alpha)ds,
\end{eqnarray*}
where $\mathcal{C}$ is the positively oriented rectangle with vertices at $1-a+i$, $a+i$, $a+iT$ and $1-a+iT$. Here $a=1+L^{-1}$ and $T$ is chosen so that the distance from $T$ to the nearest $\gamma$ is $\gg L^{-1}$. It is standard that the contribution from the horizontal segments of the contour is $O_\varepsilon(yT^{1/2+\varepsilon}$).

We denote the contribution from the right edge by $\mathcal{N}_1$, where
\begin{equation}\label{100}
\mathcal{N}_1=\frac{1}{2\pi i}\int_{a+i}^{a+iT}\chi(1-s-i\alpha)\frac{\zeta'}{\zeta}(s)\zeta(s+i\alpha)^2G(s+i\alpha)G(1-s-i\alpha)ds.
\end{equation}
From the functional equation we have
\begin{equation*}
\frac{\zeta'}{\zeta}(1-s)=\frac{\chi'}{\chi}(1-s)-\frac{\zeta'}{\zeta}(s).
\end{equation*}
Hence the contribution from the left edge,  by substituting $s$ by $1-s$, is
\begin{eqnarray*}
&&\frac{1}{2\pi i}\int_{a-i}^{a-iT}\frac{\zeta'}{\zeta}(1-s)\zeta(1-s+i\alpha)\zeta(s-i\alpha)G(1-s+i\alpha)G(s-i\alpha)ds\nonumber\\
&=&\frac{1}{2\pi i}\int_{a-i}^{a-iT}\bigg(\frac{\chi'}{\chi}(1-s)-\frac{\zeta'}{\zeta}(s)\bigg)\zeta(1-s+i\alpha)\zeta(s-i\alpha)G(1-s+i\alpha)G(s-i\alpha)ds\nonumber\\
&=&-\overline{\mathcal{N}_2}+\overline{\mathcal{N}_1}+O_\varepsilon(yT^{1/2+\varepsilon}),
\end{eqnarray*}
where
\begin{equation}\label{506}
\mathcal{N}_2(\beta,\gamma)=\frac{1}{2\pi i}\int_{a+i}^{a+iT}\frac{\chi'}{\chi}(1-s)\zeta(1-s+i\alpha)\zeta(s-i\alpha)G(1-s+i\alpha)G(s-i\alpha)ds.
\end{equation}
Thus
\begin{equation}\label{806}
\sum_{0<\gamma\leq T}H(\rho+i\alpha)H(1-\rho-i\alpha)=2\textrm{Re}\big(\mathcal{N}_1\big)-\overline{\mathcal{N}_2}+O_\varepsilon(yT^{1/2+\varepsilon}).
\end{equation}

\subsection{Evaluate $\mathcal{N}_2$}

We move the line of integration in \eqref{506} to the $\tfrac{1}{2}$-line. As before, this produces an error of size $O_\varepsilon(yT^{1/2+\varepsilon})$. Hence we get
\begin{equation*}
\mathcal{N}_2=\frac{1}{2\pi}\int_{1-\alpha}^{T-\alpha}\frac{\chi'}{\chi}\big(\tfrac{1}{2}-it-i\alpha\big)\big|H(\tfrac{1}{2}+it)\big|^2dt+O_\varepsilon(yT^{1/2+\varepsilon}).
\end{equation*}

From Stirling's approximation we have
\begin{displaymath}
\frac{\chi'}{\chi}(\tfrac{1}{2}-it)=-\log\frac{t}{2\pi}+O(t^{-1})\qquad(t\geq 1).
\end{displaymath}
Combining this with Lemma 1.1 and integration by parts, we easily obtain
\begin{equation}\label{805}
\mathcal{N}_2=-\frac{ATL(\log y)^9}{12\pi}\int_{0}^{1}(1-x)^3\bigg(\vartheta^{-1}P_1(x)^2-2P_1(x)P_2(x)\bigg)dx+O(TL^9).
\end{equation}

\subsection{Evaluate $\mathcal{N}_1$}

It is easier to start with a more general sum
\begin{eqnarray*}
\mathcal{N}_1(\beta,\gamma)&=&\frac{1}{2\pi i}\int_{a+i(1+\alpha)}^{a+i(T+\alpha)}\chi(1-s)\bigg(\frac{\zeta'}{\zeta}(s+\beta)\zeta(s+\gamma)\zeta(s)\sum_{m\leq y}\frac{h(m)P([m]_y)}{m^{s}}\bigg)\\
&&\qquad\qquad\qquad\qquad\bigg(\sum_{n\leq y}\frac{h(n)P([n]_y)}{n^{1-s}}\bigg)ds,
\end{eqnarray*}
so that $\mathcal{N}_1=\mathcal{N}_1(-i\alpha,0)$. From Lemma \ref{501}, we obtain 
\begin{equation*}
\mathcal{N}_1(\beta,\gamma)=\sum_{n\leq y}\frac{h(n)P([n]_y)}{n}\sum_{m\leq nT/2\pi}a(m)e\bigg(-\frac{m}{n}\bigg)+O_\varepsilon(yT^{1/2+\varepsilon}),
\end{equation*}
where the arithmetic function $a(m)$ is defined by
\begin{equation}\label{1}
\frac{\zeta'}{\zeta}(s+\beta)\zeta(s+\gamma)\zeta(s)\sum_{m\leq y}\frac{h(m)P([m]_y)}{m^{s}}=\sum_{m=1}^{\infty}\frac{a(m)}{m^s}.
\end{equation}
By the work of Conrey, Ghosh and Gonek [\textbf{\ref{CGG1}}; Sections 5--6 and (8.2)], and the work of Bui and Heath-Brown [\textbf{\ref{BH-B}}], we can write 
\begin{equation*}
\mathcal{N}_1(\beta,\gamma) =\mathcal{Q}(\beta,\gamma) + E+O_\varepsilon(yT^{1/2+\varepsilon}),
\end{equation*}
where
\begin{equation}\label{2}
\mathcal{Q}(\beta,\gamma)=\sum_{ln\leq y}\frac{h(ln)P([ln]_y)}{ln}\frac{\mu(n)}{\varphi(n)}\sum_{\substack{m\leq nT/2\pi\\(m,n)=1}}a(lm)
\end{equation}
and
\begin{equation*}
E\ll_{B,\varepsilon} T\mathscr{L}^{-B}+y^{1/3}T^{5/6+\varepsilon}
\end{equation*}
for any $B>0$.

Let 
\begin{equation}\label{301}
\frac{\zeta'}{\zeta}(s+\beta)\zeta(s+\gamma)\zeta(s)=\sum_{n=1}^{\infty}\frac{g(n)}{n^s}.
\end{equation}
From \eqref{1} and Lemma \ref{300} we have
\begin{equation*}
a(lm)=\sum_{\substack{l=l_1l_2\\m=m_1m_2\\l_1m_1\leq y\\(m_2,l_1)=1}} h(l_1m_1)P([l_1m_1]_y)g(l_2m_2).
\end{equation*}
Hence
\begin{eqnarray}\label{502}
\mathcal{Q}(\beta,\gamma)=\sum_{l_1l_2n\leq y}\frac{h(l_1l_2n)P([l_1l_2n]_y)}{l_1l_2n}\frac{\mu(n)}{\varphi(n)}\sum_{\substack{l_1m_1\leq y\\(m_1,n)=1}} h(l_1m_1)P([l_1m_1]_y)\sum_{\substack{m_2\leq nT/2\pi m_1\\(m_2,l_1n)=1}}g(l_2m_2).
\end{eqnarray}

\begin{lemma}
Suppose $a$ and $b$ are coprime, squarefree integers. Then we have
\begin{eqnarray*}
G(x;a,b)&:=&\sum_{\substack{n\leq x\\(n,b)=1}}g(an)\\
&=&-\frac{x^{1-\beta}}{1-\beta}\sum_{a=a_2a_3}\frac{1}{a_{2}^{\gamma}}\zeta(1-\beta+\gamma)\zeta(1-\beta)F(b,-\beta+\gamma)F(a_2b,-\beta)\\
&&+\frac{x^{1-\gamma}}{1-\gamma}\sum_{a=a_2a_3}\frac{1}{a_{2}^{\gamma}}\bigg(\frac{\zeta'}{\zeta}(1+\beta-\gamma)+\sum_{p|b}\frac{\log p}{p^{1+\beta-\gamma}-1}\bigg)\zeta(1-\gamma)F(b)F(a_2b,-\gamma)\\
&&-\frac{x^{1-\gamma}}{1-\gamma}\sum_{a=pa_2a_3}\frac{1}{p^\beta a_{2}^{\gamma}}\frac{\log p}{1-p^{-(1+\beta-\gamma)}}\zeta(1-\gamma)F(pb)F(pa_2b,-\gamma)\\
&&+x\sum_{a=a_2a_3}\frac{1}{a_{2}^{\gamma}}\bigg(\frac{\zeta'}{\zeta}(1+\beta)+\sum_{p|b}\frac{\log p}{p^{1+\beta}-1}\bigg)\zeta(1+\gamma)F(b,\gamma)F(a_2b)\\
&&-x\sum_{a=pa_2a_3}\frac{1}{p^\beta a_{2}^{\gamma}}\frac{\log p}{1-p^{-(1+\beta)}}\zeta(1+\gamma)F(pb,\gamma)F(pa_2b)\\
&&\qquad\ \ \  +O_{B,\varepsilon}\big((\log ab)^{1+\varepsilon}x(\log x)^{-B}\big).
\end{eqnarray*}
\end{lemma}
\begin{proof}
It is standard that up to an error term of size $O_{B,\varepsilon}\big((\log ab)^{1+\varepsilon}x(\log x)^{-B}\big)$ for any $B>0$, $G(x;a,b)$ is the sum of the residues at $s=1-\beta$, $s=1-\gamma$ and $s=1$ of
\begin{equation*}
\frac{x^s}{s}\sum_{(n,b)=1}\frac{g(an)}{n^s}.
\end{equation*}
Combining \eqref{301} and Lemma \ref{300}, the above expression is
\begin{eqnarray*}
&&\frac{x^s}{s}\sum_{a=a_1a_2a_3}\bigg(-\sum_{(n,b)=1}\frac{\Lambda(a_1n)}{(a_1n)^{\beta}n^s}\bigg)\bigg(\sum_{(n,a_1b)=1}\frac{1}{(a_2n)^{\gamma}n^s}\bigg)\bigg(\sum_{(n,a_1a_2b)=1}\frac{1}{n^s}\bigg)\\
&=&\frac{x^s}{s}\sum_{a=a_1a_2a_3}\frac{1}{a_{1}^{\beta}a_{2}^{\gamma}}\bigg(-\sum_{(n,b)=1}\frac{\Lambda(a_1n)}{n^{s+\beta}}\bigg)\zeta(s+\gamma)\zeta(s)F(a_1b,s+\gamma-1)F(a_1a_2b,s-1).
\end{eqnarray*}
We have
\begin{displaymath}
-\sum_{(n,b)=1}\frac{\Lambda(a_1n)}{n^{s+\beta}}=\left\{ \begin{array}{ll}
\frac{\zeta'}{\zeta}(s+\beta)+\sum_{p|b}\frac{\log p}{p^{s+\beta}-1} &\qquad \textrm{if $a_1=1$},\\
-\frac{\log p}{1-p^{-(s+\beta)}} & \qquad\textrm{if $a_1=p$,}\\
0 & \qquad\textrm{otherwise}.
\end{array} \right.
\end{displaymath}
The result follows.
\end{proof}

In view of the above definition, the innermost sum in \eqref{502} is
\begin{equation*}
G(nT/2\pi m_1;l_2,l_1n).
\end{equation*}
We then write
\begin{equation*}
\mathcal{Q}(\beta,\gamma)=\sum_{j=1}^{6}\mathcal{Q}_j(\beta,\gamma)
\end{equation*}
corresponding to the decomposition of $G(x;a,b)$ in Lemma 4.1.

We begin with $\mathcal{Q}_1(\beta,\gamma)$. Writing $l_2l_3$ for $l_2$, and $m$ for $m_1$, we have $\mathcal{Q}_1(\beta,\gamma)$ equals
\begin{eqnarray*}
&&-\frac{(T/2\pi)^{1-\beta}}{1-\beta}\zeta(1-\beta+\gamma)\zeta(1-\beta)\sum_{\substack{l_1l_2l_3\leq y\\l_1m\leq y}}\frac{h(l_1l_2l_3)h(l_1m)P([l_1m]_y)}{l_1l_{2}^{1+\gamma}l_3m^{1-\beta}}F(l_1,-\beta+\gamma)F(l_1l_2,-\beta)\\
&&\qquad\sum_{\substack{n\leq y/l_1l_2l_3\\(n,l_1l_2l_3m)=1}}\frac{\mu(n)h(n)P([l_1l_2l_3n]_y)}{\varphi(n)n^\beta}F(n,-\beta+\gamma)F(n,-\beta).
\end{eqnarray*}
From Lemma \ref{600}, the innermost sum is
\begin{eqnarray*}
&&U_1V_1(l_1l_2l_3m)\bigg(\frac{P''([l_1l_2l_3]_y)}{(\log y)^2}+\frac{2\beta P'([l_1l_2l_3]_y)}{\log y}+\beta^{2}P([l_1l_2l_3]_y)\bigg)\\
&&\qquad+O(F_{\tau_0}(l_1l_2l_3m)L^{-3})+O_\varepsilon\bigg(F_{\tau_0}(l_1l_2l_3m)\bigg(\frac{y}{l_1l_2l_3}\bigg)^{-\nu}L^{-2+\varepsilon}\bigg).
\end{eqnarray*}
By Lemma \ref{505}, the contributions of the $O$-terms to $\mathcal{Q}_1(\beta,\gamma)$ is $O_\varepsilon(TL^{9+\varepsilon})$. Hence 
\begin{eqnarray*}
&&\mathcal{Q}_1(\beta,\gamma)=-U_1(T/2\pi)^{1-\beta}\zeta(1-\beta+\gamma)\zeta(1-\beta)\sum_{l_1l_2\leq y}\frac{F(l_1,-\beta+\gamma)F(l_1l_2,-\beta)}{l_1l_{2}^{1+\gamma}}\\
&&\qquad\bigg(\frac{\mathcal{A}_2(y,P,P'';l_1,l_2,-\beta)}{(\log y)^2}+\frac{2\beta \mathcal{A}_2(y,P,P';l_1,l_2,-\beta)}{\log y}+\beta^2\mathcal{A}_2(y,P,P;l_1,l_2,-\beta)\bigg)\\
&&\qquad\qquad+O_\varepsilon(TL^{9+\varepsilon}).
\end{eqnarray*}
Using Lemmas \ref{601}--\ref{602} we obtain
\begin{eqnarray}\label{800}
&&\mathcal{Q}_1(\beta,\gamma)=-\frac{A(T/2\pi)^{1-\beta}(\log y)^{10}}{6}\zeta(1-\beta+\gamma)\zeta(1-\beta)\int_{0}^{1}\int_{0}^{x}\int_{0}^{x}(1-x)^3y^{\beta t-\gamma t_1}tt_1\nonumber\\
&&\qquad P(x-t)\bigg(\frac{P(x-t_1)}{(\log y)^2}+\frac{2\beta P_0(x-t_1)}{\log y}+\beta^2P_1(x-t_1)\bigg)dtdt_1dx+O_\varepsilon(TL^{9+\varepsilon}).
\end{eqnarray}
Here we have used a fact which is easy to verify that $U_1U_2W=A$.

For $\mathcal{Q}_2(\beta,\gamma)$, we write the sum $\sum_{p|l_1n}$ as $\sum_{p|l_1}+\sum_{p|n}$, since the function $h(n)$ is supported on square-free integers. In doing so we have $\mathcal{Q}_2(\beta,\gamma)$ equals
\begin{eqnarray}\label{605}
&&\frac{(T/2\pi)^{1-\gamma}}{1-\gamma}\zeta(1-\gamma)\sum_{\substack{l_1l_2l_3\leq y\\l_1m\leq y}}\frac{h(l_1l_2l_3)h(l_1m)P([l_1m]_y)}{l_1l_{2}^{1+\gamma}l_3m^{1-\gamma}}\bigg(\frac{\zeta'}{\zeta}(1+\beta-\gamma)+\sum_{p|l_1}\frac{\log p}{p^{1+\beta-\gamma}-1}\bigg)\nonumber\\
&&\qquad\qquad F(l_1)F(l_1l_2,-\gamma)\sum_{\substack{n\leq y/l_1l_2l_3\\(n,l_1l_2l_3m)=1}}\frac{\mu(n)h(n)P([l_1l_2l_3n]_y)}{\varphi(n)n^{\gamma}}F(n)F(n,-\gamma)\nonumber\\
&&\qquad+\frac{(T/2\pi)^{1-\gamma}}{1-\gamma}\zeta(1-\gamma)\sum_{\substack{l_1l_2l_3\leq y\\l_1m\leq y}}\frac{h(l_1l_2l_3)h(l_1m)P([l_1m]_y)}{l_1l_{2}^{1+\gamma}l_3m^{1-\gamma}}F(l_1)F(l_1l_2,-\gamma)\nonumber\\
&&\qquad\qquad\sum_{\substack{p|n\\n\leq y/l_1l_2l_3\\(n,l_1l_2l_3m)=1}}\frac{\log p}{p^{1+\beta-\gamma}-1}\frac{\mu(n)h(n)P([l_1l_2l_3n]_y)}{\varphi(n)n^{\gamma}}F(n)F(n,-\gamma).
\end{eqnarray}
We consider the contribution from the terms $\sum_{p|l_1}$. From Lemma \ref{600}, the sum over $n$ is 
\begin{equation*}
\ll L^{-2}+F_{\tau_0}(l_1l_2l_3m)L^{-3}+O_\varepsilon\bigg(F_{\tau_0}(l_1l_2l_3m)\bigg(\frac{y}{l_1l_2l_3}\bigg)^{-\nu}L^{-2+\varepsilon}\bigg).
\end{equation*}
Hence the contribution of the terms $\sum_{p|l_1}$ to $\mathcal{Q}_2(\beta,\gamma)$ is
\begin{eqnarray*}
&\ll_\varepsilon& TL^{-1}\sum_{\substack{p|l_1\\l_1l_2l_3\leq y\\l_1m\leq y}}\frac{\log p}{p-1}\frac{d_4(l_1)d(l_2)d(l_3)d(m)}{l_1l_2l_3m}\\
&&\qquad\qquad\qquad\qquad\bigg(1+F_{\tau_0}(l_1l_2l_3m)L^{-1}+F_{\tau_0}(l_1l_2l_3m)\bigg(\frac{y}{l_1l_2l_3}\bigg)^{-\nu}L^{\varepsilon}\bigg)\\
&\ll_\varepsilon&TL^{5}\sum_{\substack{p|l_1\\l_1\leq y}}\frac{\log p}{p-1}\frac{d_4(l_1)}{l_1}\big(1+F_{\tau_0}(l_1)L^{-1+\varepsilon}\big)\ll_\varepsilon TL^{9+\varepsilon}.
\end{eqnarray*}
The same argument shows that the last term in \eqref{605} is also $O_\varepsilon(TL^{9+\varepsilon})$. The remaining terms are
\begin{eqnarray*}
&&\frac{(T/2\pi)^{1-\gamma}}{1-\gamma}\frac{\zeta'}{\zeta}(1+\beta-\gamma)\zeta(1-\gamma)\sum_{\substack{l_1l_2l_3\leq y\\l_1m\leq y}}\frac{h(l_1l_2l_3)h(l_1m)P([l_1m]_y)}{l_1l_{2}^{1+\gamma}l_3m^{1-\gamma}}\nonumber\\
&&\qquad\qquad F(l_1)F(l_1l_2,-\gamma)\sum_{\substack{n\leq y/l_1l_2l_3\\(n,l_1l_2l_3m)=1}}\frac{\mu(n)h(n)P([l_1l_2l_3n]_y)}{\varphi(n)n^{\gamma}}F(n)F(n,-\gamma).
\end{eqnarray*}
Similarly to $\mathcal{Q}_1(\beta,\gamma)$, we thus obtain
\begin{eqnarray}\label{801}
&&\mathcal{Q}_2(\beta,\gamma)=\frac{A(T/2\pi)^{1-\gamma}(\log y)^{10}}{6}\frac{\zeta'}{\zeta}(1+\beta-\gamma)\zeta(1-\gamma)\int_{0}^{1}\int_{0}^{x}\int_{0}^{x}(1-x)^3y^{\gamma (t-t_1)}tt_1\nonumber\\
&&\qquad P(x-t)\bigg(\frac{P(x-t_1)}{(\log y)^2}+\frac{2\gamma P_0(x-t_1)}{\log y}+\gamma^2P_1(x-t_1)\bigg)dtdt_1dx+O_\varepsilon(TL^{9+\varepsilon}).
\end{eqnarray}

The fourth term $\mathcal{Q}_4(\beta,\gamma)$ is in the same form as $\mathcal{Q}_2(\beta,\gamma)$. The same calculations yield
\begin{eqnarray}\label{802}
\mathcal{Q}_4(\beta,\gamma)&=&\frac{A(T/2\pi)(\log y)^8}{6}\frac{\zeta'}{\zeta}(1+\beta)\zeta(1+\gamma)\int_{0}^{1}\int_{0}^{x}(1-x)^3y^{-\gamma t_1}t_1\nonumber\\
&&\qquad\qquad P_1(x)P(x-t_1)dt_1dx+O_\varepsilon(TL^{9+\varepsilon}).
\end{eqnarray}

To evaluate $\mathcal{Q}_3(\beta,\gamma)$, we rearrange the sums and write $\mathcal{Q}_3(\beta,\gamma)$ in the form
\begin{eqnarray*}
&&-\frac{(T/2\pi)^{1-\gamma}}{1-\gamma}\zeta(1-\gamma)\sum_{\substack{pl_1l_2l_3\leq y\\l_1m\leq y}}\frac{\log p}{(p^{1+\beta-\gamma}-1)p^{\gamma}}\frac{h(pl_1l_2l_3)h(l_1m)P([l_1m]_y)}{l_1l_{2}^{1+\gamma}l_3m^{1-\gamma}}\\
&&\qquad F(pl_1)F(pl_1l_2,-\gamma)\sum_{\substack{n\leq y/pl_1l_2l_3\\(n,pl_1l_2l_3m)=1}}\frac{\mu(n)h(n)P([pl_1l_2l_3n]_y)}{\varphi(n)n^{\gamma}}F(n)F(n,-\gamma).
\end{eqnarray*}
By Lemma \ref{600}, the innermost sum is
\begin{eqnarray*}
&&U_1V_1(pl_1l_2l_3m)\bigg(\frac{P''([pl_1l_2l_3]_y)}{(\log y)^2}+\frac{2\gamma P'([pl_1l_2l_3]_y)}{\log y}+\gamma^{2}P([pl_1l_2l_3]_y)\bigg)\\
&&\qquad\qquad+O(F_{\tau_0}(pl_1l_2l_3m)L^{-3})+O_\varepsilon\bigg(F_{\tau_0}(pl_1l_2l_3m)\bigg(\frac{y}{pl_1l_2l_3}\bigg)^{-\nu}L^{-2+\varepsilon}\bigg).
\end{eqnarray*}
The contribution of the $O$-terms, using Lemma \ref{505}, is $O_\varepsilon(TL^{9+\varepsilon})$. The remaining terms contribute
\begin{eqnarray*}
&&-\frac{U_1(T/2\pi)^{1-\gamma}}{(1-\gamma)}\zeta(1-\gamma)\sum_{pl_1l_2\leq y}\frac{\log p}{(p^{1+\beta-\gamma}-1)p^{\gamma}}\frac{F(pl_1)F(pl_1l_2,-\gamma)}{l_1l_{2}^{1+\gamma}}\\
&&\qquad\bigg(\frac{\mathcal{A}_2(y,P,P'';l_1,pl_2,-\gamma)}{(\log y)^2}+\frac{2\gamma \mathcal{A}_2(y,P,P';l_1,pl_2,-\gamma)}{\log y}+\gamma^2\mathcal{A}_2(y,P,P;l_1,pl_2,-\gamma)\bigg).
\end{eqnarray*}
In view of Lemma \ref{601}, this equals
\begin{eqnarray*}
&&-U_1U_2(T/2\pi)^{1-\gamma}(\log y)^4\zeta(1-\gamma)\sum_{pl_1l_2\leq y}\frac{\log p}{(p^{1+\beta-\gamma}-1)p^{\gamma}}\frac{h(pl_1l_2)h(l_1)}{l_1l_{2}^{1+\gamma}}\\
&&\qquad F(pl_1)F(pl_1l_2,-\gamma)V_1(pl_1l_2)V_2(l_1)V_3(pl_2)V_4(pl_1l_2)\int_{0}^{[l_1]_y}y^{\gamma t}tP([l_1]_y-t)\\
&&\qquad\qquad \bigg(\frac{P([pl_1l_2]_y)}{(\log y)^2}+\frac{2\gamma P_0([pl_1l_2]_y)}{\log y}+\gamma^2P_1([pl_1l_2]_y)\bigg)dt+O(TL^9).
\end{eqnarray*}
From Lemma 2.8(ii) we obtain
\begin{eqnarray}\label{803}
&&\!\!\!\!\!\!\!\!\mathcal{Q}_3(\beta,\gamma)=-\frac{A(T/2\pi)^{1-\gamma}(\log y)^{11}}{3}\zeta(1-\gamma)\int_{0}^{1}\int_{\substack{t,t_j\geq 0\\t\leq x\\t_1+t_2\leq x}}(1-x)^3y^{\gamma(t-t_1)-\beta t_2}tt_1P(x-t)\nonumber\\
&&\!\!\!\!\!\!\!\!\  \bigg(\frac{P(x-t_1-t_2)}{(\log y)^2}+\frac{2\gamma P_0(x-t_1-t_2)}{\log y}+\gamma^2P_1(x-t_1-t_2)\bigg)dtdt_1dt_2dx+O_\varepsilon(TL^{9+\varepsilon}).
\end{eqnarray}

The term $\mathcal{Q}_5(\beta,\gamma)$ is in the same form as $\mathcal{Q}_3(\beta,\gamma)$. The same calculations give
\begin{eqnarray}\label{804}
\mathcal{Q}_5(\beta,\gamma)&=&-\frac{A(T/2\pi)(\log y)^9}{3}\zeta(1+\gamma)\int_{0}^{1}\int_{\substack{t_j\geq 0\\t_1+t_2\leq x}}(1-x)^3y^{-\gamma t_1-\beta t_2}t_1\nonumber\\
&&\qquad\qquad P_1(x)P(x-t_1-t_2)dt_1dt_2dx+O_\varepsilon(TL^{9+\varepsilon}).
\end{eqnarray}

Finally, we have $\mathcal{Q}_6(\beta,\gamma)=O_B(TL^{-B})$ for any $B>0$.

Collecting the estimates \eqref{806}, \eqref{805}, \eqref{800}, \eqref{801}--\eqref{804}, and letting $\beta=- i\alpha$, $\gamma\rightarrow 0$ we easily obtain Lemma 1.2.

\end{document}